\documentclass[leqno,12pt,a4paper]{article}

\usepackage{amscd, amsmath, amssymb, amsthm, amscd, latexsym}
\usepackage[all]{xy}
\newtheorem{theorem}{Theorem}
\theoremstyle{plain}

\newtheorem{lemma}[theorem]{Lemma}

\theoremstyle{definition}

\newtheorem*{definition}{Definition}

\newtheorem*{remark}{Remark}

\newcommand{\ab}{\mathrm{ab}}
\newcommand{\C}{\mathcal{C}}
\newcommand{\Cf}{\textrm{cf.}\;}
\newcommand{\cs}{\mathrm{cs}}
\newcommand{\CX}{\C_X}

\newcommand{\fg}{\pi_1}
\newcommand{\fgX}{\fg(X)}
\newcommand{\fgXbar}{\fg(\Xbar)}
\newcommand{\fgXab}{\fgX^{\ab}}

\newcommand{\fgXbarabcs}{\fgXbar^{\ab}_{\cs}}
\newcommand{\fgXag}{\fgX^{\ab,\geo}}
\newcommand{\fgXabcs}{\fgXab_{\cs}}
\newcommand{\geo}{\mathrm{geo}}

\newcommand{\I}{\mathcal{I}}
\renewcommand{\Im}{\operatorname{Im}}
\newcommand{\inj}{\hookrightarrow}
\newcommand{\isom}{\onto{\sim}}

\newcommand{\IX}{\I_X}

\renewcommand{\O}{\mathcal{O}}

\newcommand{\OXbarx}{\O_{\Xbar,x}}
\newcommand{\OXbarxh}{\O_{\Xbar,x}^h}
\newcommand{\ol}[1]{\overline{#1}}
\newcommand{\onto}[1]{\stackrel{#1}{\to}}
\newcommand{\opcit}{\textit{op.\,cit.}}
\newcommand{\Q}{\mathbb{Q}}
\newcommand{\Qp}{\mathbb{Q}_{p}}
\newcommand{\QZ}{\Q/\Z}
\newcommand{\ssm}{\smallsetminus}
\newcommand{\Spec}{\operatorname{Spec}}
\newcommand{\tor}{\mathrm{tor}}
\newcommand{\Xbar}{\,\ol{\!X\!}\,}
\newcommand{\Xinf}{X_{\infty}}
\newcommand{\Z}{\mathbb{Z}}
\newcommand{\Zhat}{\widehat{\Z}}

\def\sn{\smallskip\noindent}
\def\cprime{$'$} \def\cprime{$'$}

\title{Class field theory for curves over $p$-adic fields}

\author{Toshiro Hiranouchi\footnote{
 Supported by the JSPS Fellowships for Young Scientists.
}}

\begin{document}

\pagenumbering{arabic}
\maketitle


\begin{abstract}
We develop class field theory of curves over $p$-adic fields 
which extends the unramified theory of S.~Saito \cite{SSaito85b}. 
The class groups 
which approximate abelian \'etale fundamental groups of such curves 
are introduced 
in the terms of algebraic $K$-groups 
by imitating G.~Wiesend's class group 
for curves over finite fields \cite{Wiesend07}. 
\end{abstract}

\section{Introduction}
\label{Introduction}
Let $X$ be a regular curve over a finite field $k$ 
with function field $K$,  
$\Xbar$ the regular compactification of $X$, 
that is the regular and proper curve which 
contains $X$ as an open subvariety, and 
$\Xinf$ the finite set of closed points in 
the boundary $\Xbar\ssm X$ of $X$. 
Class field theory describes 
the abelian \'etale fundamental group $\fgXab$ of $X$
by a topological abelian group $\CX$ 
which is called the class group. 
In terms of (Milnor) $K$-groups, 
the group $\CX$ 
is the cokernel of the map
$$
  K_1(K) \to \bigoplus_{x\in X_0}K_0(k(x)) \oplus \bigoplus_{x\in \Xinf}K_1(K_x)
$$
induced by 
the inclusion $K\inj K_x$ and 
the boundary map $K_1(K_x) \to K_0(k(x))$, 
where $k(x)$  is the residue field at $x$, 
$K_x$ is the completion of $K$ at $x$ 
and $X_0$ is the set of closed points in $X$ 
(\Cf \cite{Wiesend07}). 
The reciprocity map $\rho_X:\CX \to \fgXab$ 
is defined by 
class field theory of finite fields,  
local class field theory and 
the reciprocity law. 
It has dense image and 
the kernel is the connected component of $0$ in $\CX$. 

The aim of this note is 
to develop 
class field theory for curves over {\it local fields}. 
Here, a local field 
means a complete discrete valuation field with finite residue field. 
Let 
$X$  be a regular curve over a local field $k$ with function field $K$.
The {\it class group} $\CX$ of $X$ 
is defined to be the cokernel
of the homomorphism
$$
  K_2(K) \to \bigoplus_{x\in X_0} K_1(k(x)) \oplus \bigoplus_{x\in \Xinf}K_2(K_x),
$$ 
induced by 
the inclusion  $K \inj K_x$ and  
the boundary map $K_2(K_x) \to K_1(k(x))$ 
(see Def.~\ref{def:class group} for the precise definition). 
Note that 
the residue field $k(x)$ at $x\in X_0$  
is a local field, and  
$K_x$ is a $2$-dimensional local field 
in the sense of K.~Kato, 
that is 
a complete discrete valuation field 
whose residue field is a local field. 
Next, a canonical continuous homomorphism 
$\sigma_X : \CX \to \fgXab$ 
shall be defined by 
local class field theory, 
$2$-dimensional local class field theory \cite{Kato79} and 
the reciprocity law due to S.~Saito \cite{SSaito85b}. 
Our main result is the following 
determination of its kernel and cokernel 
when the characteristic of $k$ is $0$.

\begin{theorem} 
\label{thm:main}
Let $X$  be a 
regular and geometrically connected 
curve over a finite extension $k$ of $\Qp$. 

\sn
$\mathrm{(i)}$  The kernel of $\sigma_X$ 
is the maximal divisible subgroup of  $\CX$.
%

\sn
$\mathrm{(ii)}$ 
The quotient of  $\fgXab$ by the topological closure 
$\ol{\Im(\sigma_X)}$
of the image of $\sigma_X$ is isomorphic to  $\Zhat^{r}$ 
with some  $r\ge 0$.
\end{theorem}
  Further assume that the variety $X$ is proper.  
  In this case, 
  the class group $\CX$ is nothing other than $SK_1(X)$.  
  By using this,  
  S.~Saito \cite{SSaito85b}
  showed the above theorem 
  and it plays an important role 
  in higher dimensional class field theory 
  of K.~Kato and S.~Saito. 
  The invariant $r$ in the above theorem 
  is called the {\it rank}\/ of the compactification $\Xbar$ of $X$
  (\opcit, Def.~2.5). 
  It depends on the type of the reduction of $\Xbar$. 
  In particular, we have $r=0$  if it has potentially good reduction. 
\begin{remark}
%
  As in \opcit, 
  for a local field $k$ with characteristic $p>0$, 
  the theorem above can be proved 
  with restriction to ``the prime-to-$p$ part'' 
  in the assertion (i).
\end{remark}

After introducing the class group of $X$ 
and the reciprocity map in Section~\ref{sec:class groups}, 
we shall prove Theorem~\ref{thm:main} 
in Section~\ref{sec:proof}.
 
Throughout this paper, 
a {\it curve}\/ over a field is 
an integral separated scheme of finite type over the field 
of dimension 1.  
For an abelian group $A$,  
we denote by $A/n$ 
the cokernel of the map $n : A \to A$ 
defined by $x\mapsto nx$ for any positive integer $n$. 

\medskip\noindent
{\it Acknowledgments.} 
The author is indebted to  
Shuji Saito for his encouragement 
during author's visit to the University of Tokyo
and helpful suggestions on 
the proof of Theorem~\ref{thm:main}. 

\section{Class Groups}
\label{sec:class groups}
Let 
$X$ be a regular curve over a local field $k$ with function field $K$,  
$\Xbar$ the regular compactification of $X$, and 
$\Xinf$ the finite set of closed points in the boundary $\Xbar\ssm X$ of $X$.
We define a group  $\IX$  by 
$$
\IX = \bigoplus_{x \in X_0} K_1(k(x)) 
    \oplus \bigoplus_{x\in \Xinf}K_2(K_x).
$$
The topology of $K_2(K_x)$ is defined in \cite{Kato79} 
(\Cf \opcit, I, Sect.~7). 
In particular, if the characteristic of $k$ is $0$, 
we take the discrete topologies of
$K_x^{\times}$  and $K_2(K_x)$. 
The group $\IX$ is endowed with the direct sum topology, 
that is, 
a subset is open if the intersection with each finite partial sum is open. 

\begin{definition}
\label{def:class group}
Define the {\it class group}
$\CX$  associated with  $X$  by the cokernel of 
the natural map $K_2(K) \to \IX$  
which is 
defined by the boundary map $K_2(K_x) \to K_1(k(x))$ 
and the inclusion $K \inj K_x$. 
The quotient topology makes this an abelian topological group. 
\end{definition}

The {\it reciprocity map} 
$$
  \sigma_X : \CX \to \fgXab
$$
is defined as follows: 
For $x \in X_0$, 
the reciprocity map of local class field theory 
$K_1(k(x)) \to \fg(x)^{\ab}$ and 
the natural map $\fg(x)^{\ab} \to \fgXab$ give 
$K_1(k(x)) \to \fgXab$. 
For any $x \in \Xinf$, 
the reciprocity map of 
2-dimensional local class field theory 
$K_2(K_x) \to \fg(\Spec(K_x))^{\ab}$   
and 
the natural map $\fg(\Spec(K_x))^{\ab} \to \fgXab$ 
define a map $K_2(K_x) \to \fgXab$. 
Thus, we have $\IX \to \fgXab$. 
Finally, 
the recirprocity law of $K$
(\cite{SSaito85b}, Chap.~II, Prop.~1.2)
and 2-dimensional local class field theory 
(\opcit, Chap.~II, Th.~3.1) 
show that 
the homomorphism $\IX\to \fgXab$ defined above 
factors through  $\CX$. 
Thus the required homomorphism $\sigma_X: \CX \to \fgXab$ is obtained. 

The structure map  $X\to \Spec(k)$  induces a map 
$N:\CX \to k^{\times}$ 
which is defined by norms over $k$ 
and one denotes the kernel of this map by  $V(X)$. 
It makes the following diagram commutative:
\begin{equation}
\label{eq:VX}
\vcenter{
\xymatrix{
  0 \ar[r] & V(X) \ar[r]\ar@{-->}^{\tau_X}[d] & \CX \ar[r]^N\ar[d]^{\sigma_X} & k^{\times}\ar@<-1mm>[d]^{\rho_k} \\
  0 \ar[r] & \fgXag \ar[r] & \fgXab \ar[r] & \fg(\Spec(k))^{\ab}.
}
}
\end{equation}
Here, 
the group $\fgXag$ is defined by the exactness of the lower horizontal row.
\begin{remark}
  As in \cite{Wiesend07}, 
  we can define a class group and a reciprocity map for 
  a regular {\it variety}\/ over the local field $k$. 
  More generally, 
  for a regular variety over a {\it higher dimensional local field}, 
  a class group may be defined as an abstract group 
  by using Milnor $K$-groups of higher degree. 
  However, there is no appropriate topology 
  in the $K$-groups for degree $>2$. 
\end{remark}

\section{Proof of the Theorem}
\label{sec:proof}
In this section, we shall prove Theorem~\ref{thm:main}. 
We denote by $\fgXabcs$  the quotient 
of $\fgXab$  which classifies the abelian covers of $X$  
which are completely split. 
The assertion (ii) 
is reduced to the unramified case 
(\cite{SSaito85b}, Chap.~II, Prop.~2.2, Th.~2.4) 
as follows: 
$$
  \fgXab/\ol{\Im(\sigma_X)} \simeq \fgXabcs = \fgXbarabcs \simeq \Zhat^r, 
$$
where $r$ is the rank of $\Xbar$.

The lemma below 
is used in the proof of the assertion (i) 
in an auxiliary role.
\begin{lemma}
  \label{lem:tau}
  $\mathrm{(i)}$ The image of $\tau_X$  is finite.
  
  \sn
  $\mathrm{(ii)}$ The cokernel of $\tau_X$ is 
  isomorphic to $\Zhat^r$.
\end{lemma}
\begin{proof}
Since the map $N: \CX \to k^{\times}$ 
is induced by norms over $k$, 
its image
is finite index in $k^{\times}$. 
Thus, the commutative diagram (\ref{eq:VX}) 
implies $\fgXab/\ol{\Im(\sigma_X)} \simeq \fgXag/\ol{\Im(\tau_X)}$. 
There is an exact sequence of \'etale cohomology groups 
$$
 0 \to H^1(\Xbar,\QZ) \to H^1(X,\QZ) \to \bigoplus_{x\in \Xinf}H^2_x(\Xbar,\QZ)
$$
and an isomorphism $H^2_x(\Xbar,\QZ)\simeq H^0(k(x),\QZ(-1))$ of finite groups. 
The abelian \'etale fundamental group has the description 
($\fg(X)^{\ab})^{\ast} \simeq H^1(X,\QZ)$, where  
the superscript ``$\ast$''  denotes the Pontrjagin dual.
Thus, Theorem~1 in \cite{Yoshida03} and the above exact sequence 
imply the following description of $\fgXag$: 
$$
  0 \to \fgXag_{\tor} \to \fgXag \to \Zhat^r \to 0,
$$
where the torsion subgroup  $\fgXag_{\tor}$  of 
$\fgXag$  is finite
(Note that, the rank of  $\Xbar$  
is the rank of the special fiber 
of the N\'eron model of the Jacobian variety of $\Xbar$, 
\Cf \cite{SSaito85b}, Chap.~II, Th.~6.2). 
Since the quotient group 
$\fgXag/\ol{\Im(\tau_X)}$ and 
$\fgXag$ are $\Zhat$-modules of rank $r$, 
the image of $\tau_X$ is finite. 
The assertions (i) and (ii) follows from it. 
\end{proof}

If we assume the following lemma, then 
the rest of the proof of the assertion (i) in Theorem~\ref{thm:main} 
is essentially 
the same as in the proof of Theorem~5.1 in Chapter~II of \cite{SSaito85b}
(by using Lem.~\ref{lem:tau}).

\begin{lemma}
\label{lem:inj}
Let  $n$  be a positive integer. 
Then the map  $\sigma_X:\CX \to \fgXab$  induces the injection
$$
  \CX/n \inj \fgXab/n.
$$
\end{lemma}
\begin{proof}
(Compare with the proof of \cite{SSaito85b}, Chap.~II, Lem.~5.3.)
  By the duality theorem of \'etale cohomology groups with compact support, 
  we have
  \begin{equation} \label{eq:pi}
    \fgXab/n = H^1(X,\Z/n)^{\ast} \simeq H^3_c(X,\Z/n(2)) = H^3(\Xbar,j_!\Z/n(2)),
  \end{equation} 
  where $j:X\inj \Xbar$  is the open immersion. 
  Let us consider the following diagram:
  $$
   \xymatrix@C=0mm{
    K_2(K)/n \ar[r]\ar[d]^{h^2_n} & \bigoplus_{x\in X_0}K_1(k(x))/n \oplus \bigoplus_{x\in \Xinf} K_2(K_x)/n  \ar[r] \ar[d]^h & \CX/n \ar[r]\ar@{-->}[d] & 0 \\
    H^2(K,\Z/n(2)) \ar[r] & \bigoplus_{x\in \Xbar_0}H^3_x(\Xbar, j_!\Z/n(2))\ar[r]&  H^3(\Xbar,j_!\Z/n(2)).  
    }
  $$
  Here, the horizontal sequences are exact, 
  and the left vertical map $h^2_n$  
  is the isomorphism by the Merkur{\cprime}ev-Suslin theorem \cite{MS82}. 
  The vertical map $h$ is an isomorphism defined as follows:
  For $x\in X_0$, 
  by excision and the purity theorem we have 
  $$
    H^3_x(\Xbar, j_!\Z/n(2)) \simeq H^3_x(X,\Z/n(2)) \simeq H^1(k(x),\Z/n (1)). 
  $$
  Thus, Kummer theory gives an isomorphism 
  \begin{equation}\label{eq:Kummer}
    K_1(k(x))/n \isom H^1(k(x),\Z/n (1)) \isom H^3(\Xbar, j_!\Z/n(2)).
  \end{equation}  
  For $x\in \Xinf$, 
  let $\OXbarxh$ be the henselization of $\OXbarx$, 
  $K_x^h$ the field of fractions of $\OXbarxh$, 
  and 
  $j_x:\Spec(K_x^h) \inj \Spec(\OXbarxh)$ the inclusion. 
  By excision and Proposition 1.1 in \cite{Milne06}, we have
  $$
    H^3_x(\Xbar, j_!\Z/n(2)) 
      \simeq H^3_x(\Spec(\OXbarxh), j_{x!}\Z/n(2)) \simeq H^2(K_x^h, \Z/n(2)).
  $$ 
  The Merkur{\cprime}ev-Suslin theorem 
  gives an isomorphism  
  $$
    K_2(K_x)/n \isom K_2(K_x^h)/n \isom H^2(K_x^h, \Z/n(2)) \isom H^3_x(\Xbar,j_!\Z/n(2)). 
  $$
  By composing this and (\ref{eq:Kummer}), 
  the isomorphism $h$ is defined. 
  From the above diagram and (\ref{eq:pi}), we obtain an injection 
  $$
    \CX/n \inj H^3(\Xbar,j_!\Z/n(2)) \isom \fgXab/n
  $$
  which is nothing other than the map $\sigma_X$  modulo $n$.
\end{proof}

\begin{remark}
  Q.~Tian \cite{Tian98} 
  established a similar result 
   by using relative $K$-groups $SK_1(\Xbar,D)$, 
  where $D := \Xbar\ssm X$ is the reduced Weil divisor on $\Xbar$. 
  However, it seems that 
  the theorem (\opcit, Th.~3.11) 
  corresponding to the lemma above 
  is not proved completely.  
\end{remark}


\providecommand{\bysame}{\leavevmode\hbox to3em{\hrulefill}\thinspace}
\providecommand{\MR}{\relax\ifhmode\unskip\space\fi MR }
\providecommand{\MRhref}[2]{%
  \href{http://www.ams.org/mathscinet-getitem?mr=#1}{#2}
}
\providecommand{\href}[2]{#2}


\vspace{0.5cm}

\noindent
 Toshiro Hiranouchi \\
 Graduate School of Mathematics \\
 Kyushu University \\
 6-10-1, Hakozaki, Higashiku, Fukuoka-city, 812-8581
 Japan \\
JSPS Research Fellow, \\
{\tt hiranouchi@math.kyushu-u.ac.jp}

\end{document}